\documentclass[12pt,oneside]{amsart}
\usepackage{amscd,amsmath,amssymb,amsfonts}
\usepackage[all]{xy}
\usepackage{hyperref}
\usepackage{url}
\usepackage{stmaryrd}

\usepackage[a4paper, total={6in, 8in}]{geometry}

\usepackage{xcolor}

\usepackage{upgreek}

\theoremstyle{plain}
\newtheorem{thm}{Theorem}
\newtheorem{lem}[thm]{Lemma}
\newtheorem{cor}[thm]{Corollary}
\newtheorem{prop}[thm]{Proposition}

\theoremstyle{definition}

\newtheorem{defn}[thm]{Definition}

\newtheorem{question}[thm]{Question}

\newtheorem{rmks}[thm]{Remarks}
\numberwithin{thm}{section}
\numberwithin{equation}{section}

\newcommand{\C}{{\mathbb C}}

\newcommand{\R}{{\mathbb R}}

\newcommand{\Z}{{\mathbb Z}}

\begin{document}

\title{Algebraic flat connections and o-minimality}
\author{H\'el\`ene Esnault \and Moritz Kerz}
 \address{Freie Universit\"at Berlin, Berlin,  Germany;
 }
 \email{esnault@math.fu-berlin.de }
 \address{   Fakult\"at f\"ur Mathematik \\
 Universit\"at Regensburg \\
 93040 Regensburg, Germany}
 \email{moritz.kerz@mathematik.uni-regensburg.de}
 \thanks{ The second author is supported by the SFB 1085 Higher Invariants, Universit\"at Regensburg}
\dedicatory{\`A G\'erard Laumon, en t\'emoignage d'amiti\'e}

\begin{abstract}
  We prove that an algebraic flat connection has $\mathbb R_{\rm an,exp}$-definable flat
  sections if and only if it is regular singular with unitary monodromy eigenvalues at
  infinity, refining previous work of Bakker–Mullane. This provides an o-minimal
  characterization of classical properties of the Gauss-Manin connection.
\end{abstract}

\maketitle

\section{Introduction}\label{sec:intro}

Let $X$ be a smooth complex algebraic variety.
Let $\mathcal E$ be an algebraic vector bundle on $X$ and $\nabla\colon \mathcal E\to
\Omega^1_X\otimes_{\mathcal O_X} \mathcal E$ be a flat algebraic connection.
We denote the analytic spaces underlying $X$ and $\mathcal E$ by the same symbol and we
endow them with their canonical $\mathbb R_{\rm an,exp}$-definable structure,
see~\cite[Sec.~2 and 3]{BBT23}. Roughly speaking, a function on $X$ is definable in this structure if it
can be expressed in terms of the usual logical operators (including  quantifiers),  algebraic
functions,  real analytic functions on compact domains like $[0,1]^n$ and the real
exponential function. A fundamental property of this structure is its o-minimality, which essentially means that definable functions don't exhibit oscillatory behavior.

We say that $(\mathcal E,\nabla)$ has {\it definable flat sections} if  any flat section
$\sigma\colon U\to \mathcal E$ is definable for any definable open
$U\subset X$.

\begin{thm}\label{thm:mainglob}
The algebraic flat bundle $(\mathcal E,\nabla)$ has definable flat sections if and only if
it is regular singular with unitary monodromy eigenvalues at infinity.
\end{thm}

\begin{rmks}
  \begin{itemize}
  \item[(1)]
    The `if' part is due to Bakker--Mullane~\cite[Thm.~1.2]{BM23}  and the `only if' part refines their Example~3.3.
  \item[(2)]
One can easily show that the property of definable flat sections holds for all definable open subsets if and only if it holds for a single finite covering by simply connected open definable subsets (see the proof of Lemma~\ref{lem:redpolysec}). In addition, all flat sections on a given definable $U$ are definable if and only  this property is true on a basis of such.
\item[(3)]
The property of unitary monodromy eigenvalues at infinity can be defined in two equivalent ways:
\begin{itemize}
\item[(i)] By considering a normal crossings compactification
$\overline X$ of $X$ and requiring that the eigenvalues of the monodromy action of a small
loop around any irreducible component of $\overline X\setminus X$ have absolute value one
(see~\cite[Def.~1.1]{BM23}).
\item[(ii)] By requiring that for all non-singular complex curves $C$ and all morphisms
  $\psi\colon C\to X$ the flat connection $\psi^*(\mathcal E,\nabla)$ has monodromy
  eigenvalues of absolute value one at each point of  the   smooth  compactification
$\overline C$ of $C$.
\end{itemize}
\item[(4)] For the definition of regular singular algebraic
  connections see~\cite[Def.~4.5]{Del70}.
\end{itemize}

\end{rmks}

The `only if' part of Theorem~\ref{thm:mainglob} is shown by reducing to a local problem in
one variable in Section~\ref{sec:reducsec} and using the solution theory of irregular singular complex differential
equations in Section~\ref{sec:onevariable}. More precisely, we use the Multisummation Theorem, recalled in Section~\ref{sec:remdiffeq}, and an o-minimal expansion of
$\mathbb R_{\rm an,exp}$ by  `multisums' due to van den Dries and
Speissegger~\cite{vdDS00}, recalled in Section~\ref{sec:remomin}.

\smallskip

Our motivation for working out a proof of Theorem~\ref{thm:mainglob} comes from the study
of the Gauss-Manin connection. In fact, Theorem~\ref{thm:mainglob} allows us to reformulate
Griffiths' theorem that the Gauss-Manin connection is regular singular \cite[Thm.~(4.3)]{Gri70}, together
with the Griffiths--Landman--Grothendieck  Monodromy
  Theorem~\cite[Thm.~(3.1)]{Gri70}. Those two theorems can now be expressed equivalently
  in terms of~\cite[Cor.~1.3]{BM23}:
  \begin{quote}
$(\star)$ \ \    The
algebraic de Rham cohomology together with the Gauss-Manin connection
$(\mathcal H_{\rm dR}^j(Y/X),\nabla)$ of a smooth projective family $ Y\to X$ of
complex algebraic varieties has definable flat sections.
\end{quote}

Observe that  the Gauss-Manin local system is defined over $\Z$, so the eigenvalues of the
local monodromies are of absolute value one if and only if they are roots of unity by Kronecker's theorem.

In non-abelian Hodge theory, Simpson has suggested variants of these classical results from
Hodge theory in terms of a logarithmic extension of the moduli space of vector bundles with
relative flat connections
$\mathrm{M}_{\rm dR}(Y/X)\to X$, see~\cite[Sec.~8]{Sim97}. In a forthcoming paper we plan to address the
\begin{question}
Which flat sections of $\mathrm{M}_{\rm dR}(Y/X)\to X$ over open definable subsets of $X$
are definable?
\end{question}

In the non-linear setting of non-abelian Hodge theory it is not clear to us 
whether all flat sections are definable
or whether there is any direct connection to Simpson's {regularity result 
for the non-abelian
Gauss-Manin connection \cite[Sec.~8]{Sim97} in the spirit of Theorem~\ref{thm:mainglob}.

\bigskip

 {\em Acknowledgments}. It is a pleasure to acknowledge discussions around the topic of this note  with Benjamin Bakker,  Philip Engel, Bruno Klingler and Salim Tayou, and to thank the referee for a careful reading.

\section{Reminder on linear meromorphic differential equations}\label{sec:remdiffeq}

Let $\rho>0$.
We consider the meromorphic differential equation
\begin{equation}\label{eq:diffeq1}
z Y' = A(z) Y
\end{equation}
on the open  disc $\Delta (\rho)=\{z \in \C, |z|<\rho\}$. Here $ A $ is an $r\times r$-matrix with entries holomorphic
functions on the punctured disc $\Delta (\rho)^\times=\Delta(\rho)\setminus \{0\}$ which
are meromorphic along $0$.  In order to exclude undesirable behavior at the
  boundary of  $\Delta (\rho)$, we assume that $A$ is
overconvergent, i.e.\ that all its entries extend to holomorphic functions on $\Delta(\tilde
\rho)^\times $ for some $\tilde \rho>\rho$.

Recall that a gauge transformation by $P$
changes $A$ to the coefficient matrix $P^{-1} A P - zP^{-1} P'$ in~\eqref{eq:diffeq1}. A pullback along the
finite covering  $z\mapsto
z^m$ for $m\in \mathbb Z_{\ge 1}$ replaces $A(z)$ by $ m A(z^m)$.  We say that $A$ is in
{\it  split
normal form} if
there exist $Q_1(z),\ldots , Q_\ell(z) \in \frac 1 z \C [\frac 1 z ]$ and constant square  matrices
$B_1,\ldots , B_\ell$    of size $r_1,\ldots, r_\ell$ with $\sum_{i=1}^\ell  r_i=r$  such that $A$ is block diagonal with blocks $Q_1 {\rm Id}_{r_1} + B_1 , \ldots ,
Q_\ell {\rm Id}_{r_\ell}+
B_\ell$.
Note that for such a split normal form $A$ a fundamental matrix solution on a sector has the shape
\begin{equation}\label{eq:stareq}
\mathcal Y(z) = \mathrm{Diag}( z^{B_1} e^{\int \frac{Q_1(z)}{z}\, dz}, \ldots , z^{B_\ell} e^{\int \frac{Q_\ell(z)}{z}\, dz} ).
\end{equation}

\begin{thm}[Fabry-Hukuhara-Turrittin-Levelt]
\label{thm:formaldiffeq}
After pullback along a finite covering there exists a formal gauge transformation $P\in
\mathrm{GL}_r\left( \C (\! (z)\! ) \right)$ which transforms the differential equation~\eqref{eq:diffeq1} into a
split normal form.
\end{thm}

By first performing a meromorphic gauge transformation on $A$, we can assume without loss
of generality that the formal gauge transformation $P$ in Theorem~\ref{thm:formaldiffeq}
is in $\mathrm{GL}_r( \C \llbracket z \rrbracket )$ and that $P(0)$ is the identity matrix. Note that the matrix $P$ is itself a solution of
a meromorphic differential equation
\begin{equation}\label{eq:diffeqgauge}
z P' =    A(z) P - P \tilde A(z)
\end{equation}
where $\tilde A$ is in split normal form.

Let $d\in S^1=\{ z\in \C\, |\, | z |=1 \}$ be a direction and
  $$D_d(\rho)=\{x d\, |\, x\in (0,\rho)\} \subset \Delta(\rho)$$
be its associated ray. By a {\it  solution of~\eqref{eq:diffeq1} on the ray} $D_d(\rho)$ we mean a
function $Y\colon D_d(\rho)\to \C^r$ which extends to a holomorphic function on a sector  in $\Delta(\rho)$
containing $D_d(\rho)$ which is a solution of~\eqref{eq:diffeq1} on that sector.

\begin{defn}\label{def:multisum}
We call a function $f\colon D_d(\rho)\to \C$ a `{\it multisum}' if
\begin{itemize}
\item $f$
  is the restriction to   $D_d(\rho)$  of a real
  analytic function defined on  $ D_d(\rho')$ for some $\rho'>\rho$;
 \item for some $0<\tilde\rho \le \rho$ the function $f|_{D_d(\tilde \rho)}$ is of the
form
$$f=f_1|_{D_d(\tilde \rho)} + \cdots + f_n|_{D_d(\tilde\rho)},$$ where $f_j$ is a holomorphic function on a sector $\{ z\in
\Delta (\tilde \rho)^\times \, |\, |\mathrm{arg}(z/d) |< \kappa_j \phi_j \}$ with $\kappa_j\in
(0,1)$, $\phi_j\in (\frac \pi 2 , \pi)$ for which there exists $C>0$ such that
\[
\left| \frac{1}{n!} f_j^{(n)}(z)\right| \le C^n (n!)^{\kappa_j}
\]
for all $n\ge 0$ and  all $z$ in  the above sector;
\item  the Taylor coefficients   $\lim_{z\to 0}
f_j^{(m)}(z)$ exist for the above $f_j$ for all $m\ge 0$.
\end{itemize}
\end{defn}

The set of all `multisums' is a $\C$-algebra closed under $\frac{d}{dz}$ and a `multisum' is
uniquely characterized by its asymptotic Taylor coefficients at $z=0$.  An important point is that the Taylor series at $z=0$  might be not convergent.

A formal power series in $\C\llbracket z \rrbracket$ is {\it multisummable} in direction $d$ if
it is the Taylor series at $z=0$ of a `multisum' on $D_d(\rho)$ for some $\rho >0$. A
basic reference for multisummable power series is~\cite[Ch.~10]{Bal94}.

Recall that the {\it slopes} of the differential equation \eqref{eq:diffeq1} may be defined as $$\{ \frac{1}{m} {\rm deg}_z  Q_i(\frac{1}{z}), i=1,\ldots, \ell \}$$ where  $m$ is the degree of the necessary covering to achieve the split normal form in Theorem~\ref{thm:formaldiffeq}. See~\cite[Ch.~3]{vdPS03}.

\begin{thm}[Multisummation Theorem]\label{thm:multisum}
  Let $Y\in (\C\llbracket z \rrbracket)^r $ be a formal solution of a differential
  equation~\eqref{eq:diffeq1} all of whose positive slopes are $>\frac 1 2$. Then for all but a finite number of directions $d$,  the
  components of $Y$ are multisummable in the direction $d$.
\end{thm}

For a proof of the Multisummation Theorem see for instance~\cite[Ch.~7]{vdPS03}.

\begin{rmks}
  \begin{itemize}
    \item The condition on the slopes in Theorem~\ref{thm:multisum} is satisfied after
      pullback along a finite covering $z\mapsto z^m$ for  large  $m$.
    \item
Note that the uniqueness of the `multisum' with given aymptotic Taylor expansion implies that the associated `multisums' in
Theorem~\ref{thm:multisum} automatically satisfy the differential
  equation~\eqref{eq:diffeq1} and can be defined on $D_d(\rho)$ for the given $\rho$
  in~\eqref{eq:diffeq1}.   \end{itemize}
\end{rmks}

Combining Theorem~\ref{thm:formaldiffeq} for the differential equation~\eqref{eq:diffeq1} and
Theorem~\ref{thm:multisum} for the differential equation~\eqref{eq:diffeqgauge} with the above
remarks we deduce:

\begin{cor}\label{cor:irrsingdiffeq}
  After taking the pullback along a finite covering $z\mapsto z^m$ for some $m$  and performing a meromorphic
  gauge transformation,
  there exist  $Q_1(z),\ldots , Q_L(z) \in \frac 1 z \C [\frac 1 z ]$ and constant Jordan blocks
  $B_1,\ldots , B_L$ with eigenvalues $b_1,\ldots , b_L$ such that, except for a
  finite number of directions $d$, there
  exists $P_d\colon D_d(\rho)\to \mathrm{GL}_r(\C)$ with  entries being `multisums', with
  $P_d(0)$ the identity matrix, such that
  \[
\mathcal Y\colon  D_d(\rho) \to   \mathrm{GL}_r(\C),\quad  z\mapsto P_d(z) \, \mathrm{Diag}\left(z^{B_1}  e^{\int \frac{Q_1(z)}{z}\, dz}, \ldots , z^{B_L}  e^{\int \frac{Q_L(z)}{z}\, dz} \right)
\]
is a fundamental matrix solution of the differential
  equation~\eqref{eq:diffeq1} on   $D_d(\rho)$.

The differential equation is regular singular  if and only if  $Q_1(z),\ldots , Q_L(z)$ all
vanish. In that case \[ \exp(2\pi i b_1),\ldots,  \exp(2\pi i b_L)\] are the eigenvalues
of the monodromy.
\end{cor}
Note here there is a change of notation as compared to~\eqref{eq:stareq},   $\ell$ becomes $L$. This is because we request the $B_i$ to be Jordan blocks, which is achieved after a constant gauge transformation starting from the $B_i$  and  $ Q_i(z)$ as in~\eqref{eq:stareq}.  So in fact $L\ge \ell$.

The last part of the corollary follows from~\cite[Thm.~II.1.17]{Del70}.

\section{Reminder on o-minimal structures}\label{sec:remomin}

Recall that an o-minimal structure $(\mathcal A_n)_n$ consists of a  set of subsets $A_n$
of $\mathbb R^n$ which satisfy certain properties, in particular they comprise the
semi-algebraic sets and $\mathcal A_1$ consists of the finite unions of intervals.  In this
note if not explicitly mentioned otherwise we use the o-minimal structure $\mathbb R_{\rm an,exp}$ which is the structure generated by
the semi-algebraic sets, the graphs of real analytic functions on $[0,1]^n$ and the graph
of $\exp\colon \mathbb R\to \mathbb R$.
The key property of o-minimal structures that we use is that a real analytic definable function $f\colon
(0,1)\to \mathbb R$ which is not identically zero has only a finite number of zeros. A
basic reference for o-minimal structures is~\cite{vdD98}.

Recall that a complex algebraic variety $X$ has a canonical
$\mathbb R_{\rm an,exp}$-structure (in fact a semi-algebraic structure),
see~\cite{BBT23}.

By an {\it overconvergent flat bundle $(\mathcal E,\nabla)$ on
$U(\rho)=(\Delta(\rho)^\times)^n \times \Delta(\rho)^{d-n}$ meromorphic along }
$D(\rho)=\Delta(\rho)^d\setminus U(\rho)$ we mean  the  restriction to $U(\rho)$ of  a
flat bundle on $U(\tilde \rho)$ for some
$\tilde \rho>\rho$ with a meromorphic structure along $D(\tilde \rho)$ in the sense
of~\cite[II.2.13]{Del70},  i.e.\ it is provided with a prolongation to a coherent sheaf of $\mathcal
M_{U(\tilde \rho),D(\tilde \rho)}$-modules with flat connection, where $\mathcal M_{U(\tilde \rho),D(\tilde \rho)}$
is the sheaf of rings of  holomorphic functions on $U(\tilde \rho)$ which are meromorphic  along $D(\tilde \rho)$.
 Such a $\mathcal E$
has a canonical definable structure (\cite[Intro.]{BM23}).

In a key step in our argument we use an o-minimal expansion $\mathbb R_{\rm sum,exp}$ of the structure $\mathbb R_{\rm
  an,exp}$ to a structure comprising the graphs of the `multisums' of
Definition~\ref{def:multisum}. The o-minimality of this structure is shown
in~\cite{vdDS00}, where also `multisums' in several variables are discussed.

\medskip

The elementary technical observation about oscillating functions that we need in the proof
of our main theorem is the
following. For a complex number $z$ we write $z={\rm Re}(z)+i {\rm Im}(z)$ with $ {\rm Re}(z), {\rm Im}(z)\in \R$.

\begin{lem}\label{lem:non-defin}
Consider  $Q(z)=q_k z^{-k} + q_{k-1} z^{-k+1} + \cdots + q_1z^{-1} \in \frac 1 z \C[ \frac 1 z ] $, $b\in
\C$, $\rho>0$ and $d\in S^1$. If  the function
\[
f\colon D_d(\rho) \to \C,\quad z\mapsto \exp( i\, \mathrm{Im}(Q(z) + b \log z))
\]
is definable, where $\log z$ is some branch of the logarithm, then $Q( D_d(\rho) )\subset \R$ and
$b\in \R$. In particular, if $q_k\ne 0$ then $\mathrm{Im}(q_k d^{-k})=0$.
\end{lem}

\begin{proof}

  If $Q(z)\notin \R$ for some $z\in  D_d(\rho) $ then
  $$
  \lim_{\stackrel{z\in D_d(\rho)}{z\to 0}} \left| \mathrm{ Im} \ Q(z) \right| =\infty $$
  and  grows faster  than $|{\rm log}(z)|$.
  Thus
  \[
\lim_{\stackrel{z\in D_d(\rho)}{z\to 0}} \left| \mathrm{Im}(Q(z) + b \log z)
\right| =\infty .
\]
As the real valued function  $ \mathrm{Im}(Q(z) + b \log z)$ is continuous, by the intermediate function theorem it must
take all the values $2\pi M$ for $M \ge M_0, M\in \Z$ or  $M\le M_0, M\in \Z$ for some $M_0$.
 We conclude that  $f(z)=1$ for infinitely many   $z\in
D_d(\rho) $, thus $f$ oscillates.  This contradicts definability.

Thus $Q(z)\in \R$ and in the formula for $f$ we can then omit $Q$ and the same argument shows that in view of
\[ \lim_{\stackrel{z\in D_d(\rho)}{z\to 0}} |\mathrm{Re}(\log z)|=\infty\] and
$\mathrm{Im}(\log z)$ constant the imaginary part
of $b$ has to vanish.
\end{proof}

In the opposite direction we use the following simple observation.

\begin{lem}\label{lem:defzB}
Let $B$ be a complex $n\times n$-matrix with real eigenvalues. Then on any open  sector
$S\subset\C^\times$ any holomorphic branch of the function $z\mapsto z^B$ is definable.
\end{lem}

\begin{proof}
 Up to replacing $B$ by a conjugate matrix and correspondingly $z^B$ by the same conjugate, the Jordan--Chevalley decomposition tells us
 that $B = D+N$, where $DN=ND$, $D$ is real and diagonal and $N $ is nilpotent. Then
 \[
   z^B= \exp(D \log z) \exp(N \log z),
 \]
 where the factor $\exp(N \log z)$ has entries
 which are polynomials in $\log z= \log |z| + i \, \mathrm{arg}(z)$ and where  $ \exp(D
 \log z)$ can be expressed in terms of the real exponential and logarithm functions and an analytic
 function in $ \mathrm{arg}(z)$.
 Thus  $z^B$ is
  definable.
\end{proof}

\section{Reduction to sectors}\label{sec:reducsec}

In this section we reduce the proof of Theorem~\ref{thm:mainglob} to a local
statement. Based on this reduction we then recall the proof of the `if' part due to
Bakker--Mullane~\cite{BM23}. For the `only if' part we reduce further to a local statement in one variable.

\smallskip

Let $\overline X$ be a smooth compactification of $X$ with $\overline X\setminus X$ a
normal crossings divisor. Let $\phi\colon \Delta(\rho)^d\to \overline X$ be a holomorphic chart
with $\phi^{-1}(X)= (\Delta(\rho)^\times)^n\times \Delta(\rho)^{d-n}$. Assume that $\phi$ is
overconvergent, i.e.\ extends to a holomorphic map defined on $\Delta(\rho+\epsilon)^d$ for some
$\epsilon >0$.
Let $ S=
S_1\times\cdots \times S_n \times \Delta(\rho)^{n-d}$ be an open, simply connected polysector in
$\Delta(\rho)^d$.

\begin{lem}\label{lem:redpolysec}
The algebraic flat bundle $(\mathcal E,\nabla)$ has definable flat sections
if and only if  for any polysector $S$ as above, the flat sections of $\mathcal E|_S$ are definable.
\end{lem}

\begin{proof}
We only have to show ``$\Leftarrow$''. Let $U\subset X$ be an arbitrary definable open
subset and $\sigma\colon U\to \mathcal E|_U$ a flat section. Let $\overline U\subset\overline X$ be the  closure of $U$. Then there are finitely
many charts $\phi$ as above covering $\overline U$ and finitely many polysectors  $S$ in these
charts which  cover $U$. By definability  of $U$, for any such polysector $S$,  the number of connected
components of $U\cap S$ is finite~\cite[Sec.~2.2]{vdD98}. Let $S^\circ$ be such a connected component. We can
extend the flat section $\sigma|_{S^\circ}$ to a flat section of $\mathcal E|_S$ which is
definable by assumption. So also $\sigma|_{S^\circ}$ is definable. 
 The definability of $\sigma$  on each component $S^\circ$ of $U\cap S$ implies the definability of 
 $\sigma|_{U\cap S}$. 
\end{proof}

\medskip

In the following $(\mathcal E,\nabla)$ denotes an overconvergent flat bundle on
$U=(\Delta(\rho)^\times)^n \times \Delta(\rho)^{d-n}$ which is meromorphic along
$D=\Delta(\rho)^d\setminus U$ with its canonical o-minimal structure, see Section~\ref{sec:remomin}.
Let $ S=
S_1\times\cdots \times S_n \times \Delta^{d-n}$ be an open, simply connected polysector in $\Delta(\rho)^d$ and let $\mathcal Y=(Y_1, \ldots , Y_r)$ be a basis of
the flat sections over $S$. Let $\mathcal Y M_j$ with $1\le j\le n$ and with $ M_i\in
\mathrm{GL}_r(\C)$ be the analytic continuation of $\mathcal Y$
along a simple loop around the $j$-th punctured disc. The $M_1,\ldots , M_n$ are
commuting monodromy matrices.

\begin{thm} \label{thm:main}
The flat sections $\mathcal Y$ are definable for any $S$ as above  if and only if  the flat
connection $\nabla$ is regular singular and all eigenvalues of $M_1,\ldots , M_n$ are unitary.
\end{thm}

\begin{proof}
  ``$\Leftarrow$'' (Bakker--Mullane)
By assumption the eigenvalues of $M_1,\ldots , M_n$ are unitary. Let $B_1,\ldots , B_n$ be commuting  complex matrices with real
eigenvalues and
$M_j=e^{2\pi i B_j}$. Then  the entries of
\[
  \tilde{\mathcal{Y}}(z)= \mathcal Y(z)  z_1^{-B_1} \cdots z_n^{-B_n}
\]
are single valued holomorphic functions on $ (\Delta(\rho)^\times)^n \times \Delta(\rho)^{d-n}$ with
moderate growth, so they are meromorphic, in particular definable. As also  $z_j^{-B_j}$
is  definable by Lemma~\ref{lem:defzB}, so is $\mathcal{ Y}$.

\medskip

``$\Rightarrow$''  Regular  singular is checked and monodromy is calculated locally around a
non-singular point of $D$, see~\cite[Thm.~II.4.1]{Del70}, so we can assume without loss of generality that
$n=1$. With  $ \tilde{\mathcal{Y}}(z)= \mathcal Y(z)  z_1^{-B_1} $ as above we have to
check that
\[
  \mathcal{\tilde Y}(z) = \sum_{j\in \mathbb Z} z_1^j f_j(z_2 , \ldots , z_d)
\]
is meromorphic along $D$, i.e.\ $f_j\equiv 0$ for $j\ll 0$. For this define  $E_j= \{f_j=0 \} \subset
\Delta(\rho)^{d-1}$   if $f_i\not\equiv 0$ and $E_j=\varnothing$ else.  This is a meager subset of $\Delta(\rho)^{d-1}$.
So  the countable union  $\cup_j E_j$
is a meager subset of $\Delta(\rho)^{d-1}$ as well. Thus there exists   $(\tilde z_2,\ldots , \tilde z_d)\in
\Delta(\rho)^{d-1}\setminus E$. Then if $\mathcal{\tilde Y}(z)$ is not meromorphic along
$D$, the one variable function $z_1\mapsto \mathcal{\tilde Y}(z_1,\tilde z_2 , \ldots ,
\tilde z_d)$ is not meromorphic either. This reduces us to
the case $d=1$ for checking regularity. The monodromy $M_1$ is also calculated by
restriction to these one-dimensional discs. So the implication follows from Proposition~\ref{prop:maindimone}
\end{proof}

\section{Local one variable case}\label{sec:onevariable}

 Our key ingredient in analyzing the local one variable case is the theory of
  irregular singular meromorphic differential equations and the o-minimal expansion
  $\mathbb R_{\rm sum,exp}$ of  $\mathbb R_{\rm an,exp}$ involving the multisummable
  functions due to~\cite{vdDS00}. 

In this section $(\mathcal E,\nabla)$ denotes an overconvergent flat bundle on
$\Delta(\rho)^\times$ meromorphic along $0$.
  By $d\in S^1$ we denote a direction and by $D_d(\rho):= \{ x d \, |\, x\in (0,\rho)\}
  \subset\Delta(\rho)^\times $ its associated ray.

\begin{prop}\label{prop:maindimone}
  If there are infinitely many directions $d$ such that all flat sections of $\nabla$ over $D_d(\rho)$
  are definable, then  $\nabla$ is regular singular with unitary monodromy eigenvalues.
\end{prop}

The key input in the proof is the formal and asymptotic solution theory for irregular singular differential
equations as summarized in Corollary~\ref{cor:irrsingdiffeq}.

\begin{proof}[Proof of Proposition~\ref{prop:maindimone}]
The coherent sheaf of $\mathcal M_{(\Delta(\rho)^\times, 0)}$-modules $\mathcal E$ on $\Delta(\rho)$ is  locally  free around $0$, so without loss of generality
$\mathcal E$ is the free sheaf  $\mathcal M^r_{(\Delta(\rho),0)}$. Then $(\mathcal E,\nabla)$ corresponds
to an overconvergent meromorphic differential equation~\eqref{eq:diffeq1}, see~\cite[Sec.~I.3]{Del70}.

Assume Proposition~\ref{prop:maindimone} is false,
 and  say the connection is irregular
singular.   We can take the pullback along $z\mapsto z^m$ for   large $m$
 and perform a meromorphic gauge transformation so that we are in the situation of  Corollary~\ref{cor:irrsingdiffeq}.
Assume without loss of generality that the $Q_1(z)$ in this corollary is non-zero, say
\[
  Q_1(z)=q_k z^{-k} + q_{k-1} z^{-k+1} + \cdots + q_1 z^{-1}
\]
with $q_k\neq  0$.
Choose a direction $d$ such that any solution of~\eqref{eq:diffeq1}  is $\mathbb R_{\rm an,exp}$-definable on
$D_d(\rho)$, such that we have $\mathrm{Im}(q_k d^{-k})\ne 0$ and such that the $P_d(z)$ as
in Corollary~\ref{cor:irrsingdiffeq} exists.

Then the function  $P_d(z)$ on $D_d(\rho)$
 is definable in the o-minimal structure
$\mathbb R_{\rm sum,exp}$, see Section~\ref{sec:remomin}, and so is
\[
  z\in D_d(\rho) \mapsto P_d^{-1}(  z) \mathcal Y ( z).
\]
But
multiplying the upper left entry of this matrix function with the $\mathbb R_{\rm an, exp }$-definable function
\[
D_d(\rho)\to \R,\quad  z\mapsto \exp(-\mathrm{Re} (\,Q_1(z)  +  b_1 \log(  z)) )
\]
we obtain   the  $\mathbb R_{\rm sum,exp}$-definable function
\begin{equation}\label{eq:nondeffun}
D_d(\rho)\to \C,\quad z\mapsto \exp(i \, \mathrm{Im} (\,Q_1(z) +  b_1 \log(  z)) ).
\end{equation}
However, this function is not definable in any o-minimal structure by
Lemma~\ref{lem:non-defin}, which is a contradiction. We conclude that all $Q_1(z),\ldots , Q_L(z)$ vanish.

If the monodromy eigenvalues are not
unitary, we have without loss of generality that $\mathrm{Im}(b_1)$ does not vanish
by the final part of Corollary~\ref{cor:irrsingdiffeq}
and still the function~\eqref{eq:nondeffun} would be non-definable by Lemma~\ref{lem:non-defin}, which produces the
same kind of contradiction. We conclude that all $b_1,\ldots, b_L$ are real.
\end{proof}

\end{document}